\documentclass[12pt,a4paper]{amsart}
\usepackage{amssymb,color}
\usepackage[english]{babel}
\usepackage{verbatim,here}
\usepackage[T1]{fontenc}
\usepackage{epstopdf}
\usepackage{floatflt,graphicx}
\usepackage{color,xcolor}
\usepackage{enumerate}
\usepackage{animate}
\usepackage{a4wide}
\usepackage{amsmath, amssymb}
\usepackage{amsthm}
\usepackage{hyperref}
\newcounter{minutes}
\divide\time by 60
\newcounter{hours}
\multiply\time by 60 \addtocounter{minutes}{-\time}
{\small
\curraddr{}
\email{masayo@nda.ac.jp}
\curraddr{}
\email{parisa.hariri@utu.fi}
\curraddr{}
\email{mmocanu@ub.ro}
\email{vuorinen@utu.fi}
}

\dedicatory{}
\commby{}
\theoremstyle{plain}
\newtheorem{thm}[equation]{Theorem}
\newtheorem{cor}[equation]{Corollary}
\newtheorem{lem}[equation]{Lemma}
\newtheorem{example}[equation]{Example}
\newtheorem{prop}[equation]{Proposition}

\newtheorem{problem}[equation]{Problem}
\theoremstyle{definition}

\theoremstyle{remark}
\newtheorem{rem}[equation]{Remark}

\newtheorem{nonsec}[equation]{}

\numberwithin{equation}{section}

\newcommand{\beq}{\begin{equation}}
\newcommand{\eeq}{\end{equation}}
\newcommand{\ben}{\begin{enumerate}}
\newcommand{\een}{\end{enumerate}}
\newcommand{\bequu}{\begin{eqnarray*}}
\newcommand{\eequu}{\end{eqnarray*}}
\newcommand{\bequ}{\begin{eqnarray}}
\newcommand{\eequ}{\end{eqnarray}}

\begin{document}
\def\thefootnote{}

\title[]{The ptolemy-Alhazen problem and spherical mirror reflection}
\author[M. Fujimura]{Masayo Fujimura}
\address{Department of Mathematics, National Defense Academy of Japan, Japan}
\author[P. Hariri]{Parisa Hariri}
\address{Department of Mathematics and Statistics, University of Turku,
         Turku, Finland}
\author[M. Mocanu]{Marcelina Mocanu}
\address{Department of Mathematics and Informatics, Vasile Alecsandri
         University of Bacau, Romania }
\author[M. Vuorinen]{Matti Vuorinen}
\address{Department of Mathematics and Statistics, University of Turku,
         Turku, Finland}
\date{}

\begin{abstract}
  An ancient optics problem of Ptolemy, studied later by Alhazen, is
  discussed. This problem deals with reflection of light in spherical mirrors.
  Mathematically this reduces to the solution of a quartic equation, which we
  solve and analyze using a symbolic computation software. Similar problems
  have been recently studied in connection with ray-tracing, catadioptric optics, scattering of
  electromagnetic waves, and mathematical billiards, but we were led to this
  problem in our study of the so-called triangular ratio metric.
\end{abstract}

\maketitle

\footnotetext{\texttt{{\tiny File:~\jobname .tex, printed: \number\year-%
\number\month-\number\day, \thehours.\ifnum\theminutes<10{0}\fi\theminutes}}}
\makeatletter

\makeatother



\section{Introduction}

\label{section1}

\setcounter{equation}{0}
The Greek mathematician Ptolemy (ca. 100-170)
formulated a problem concerning reflection of light at a spherical mirror
surface:
\textit{Given a light source and a spherical mirror, find the point
on the mirror where the light will be reflected to the eye of an observer.}

Alhazen (ca. 965-1040) was a scientist who lived in Iraq, Spain, and Egypt
and extensively studied several branches of science. For instance, he wrote
seven books about optics and studied e.g. Ptolemy's problem as well as many
other problems of optics and is considered to be one of the greatest
researchers of optics before Kepler \cite{a}. Often the above problem is
known as Alhazen's problem \cite[p.1010]{gbl}.

We will consider the two-dimensional version of the problem and present an
algebraic solution for it. The solution reduces to a quartic equation which
we solve with symbolic computation software.

Let $\mathbb{D}$ be the unit disk $\{ z \in \mathbb{C}: |z|<1\},$ and
suppose that the circumference $\partial \mathbb{D} = \{ z \in \mathbb{C}:
|z|=1\}$ is a reflecting curve. This two-dimensional problem reads: Given
two points $z_1,z_2 \in \mathbb{D}\,,$ find $u \in \partial \mathbb{D}$ such
that
\begin{equation}  \label{eq:theptu}
  \measuredangle (z_{1},u,0)=\measuredangle (0,u,z_{2})\,.
\end{equation}
Here $\measuredangle (z,u,w)$ denotes the
radian measure in $(-\pi ,\pi ]$ of the oriented angle with initial side
$[u,z]$ and final side $[u,w]$\,.
This equality
condition for the angles says that the angles of incidence and reflection
are equal, a light ray from $z_1$ to $u$ is reflected at $u$ and goes
through the point $z_2\,.$ Recall that, according to Fermat's principle, light travels between two points along the path that requires the least time, as compared to other nearby paths. One proves that $u=e^{it_{0}}$, $%
t_{0}\in \mathbb{R}$ satisfies (1.1) if and only if $t_{0}$ is a critical
point of the function $t\mapsto \left\vert z_{1}-e^{it}\right\vert
+\left\vert z_{2}-e^{it}\right\vert $, $t\in \mathbb{R}$. In particular,
condition (1.1) is satisfied by the extremum points (a minimum point and a
maximum point, at least) of the function $u\mapsto \left\vert
z_{1}-u\right\vert +\left\vert z_{2}-u\right\vert $, $u\in \partial \mathbb{D%
}\,.$ 

\begin{figure}[t]
  \centerline{\includegraphics[width=0.5\linewidth]{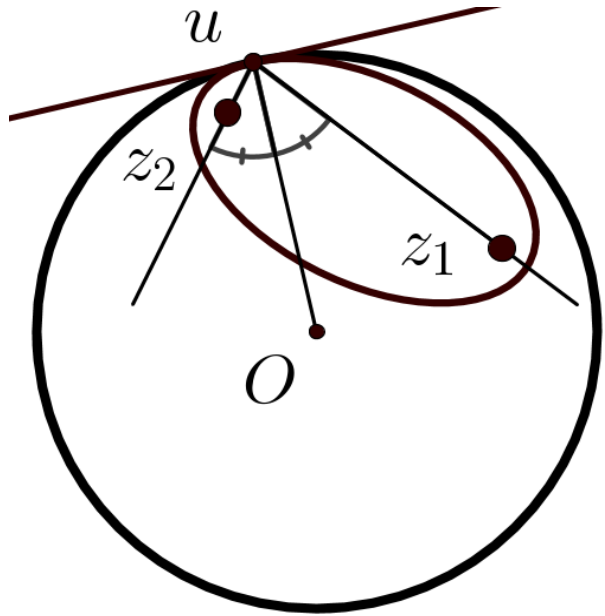}}
  \caption{Light reflection on a circular arc: The angles of incidence and
           reflection are equal. Ptolemy-Alhazen interior problem:
           Given $z_1$ and $z_2\,,$ find $u\,.$
           The maximal ellipse contained in the unit disk with foci
           $z_1$ and $z_2\,$ meets the unit circle at $u\,.$}
  \label{fig:fig1}
\end{figure}

We call this the \textit{interior problem}---there is a natural counterpart
of this problem for the case when both points are in the exterior of the
closed unit disk, called the \textit{exterior problem}. Indeed, this
exterior problem corresponds to Ptolemy's questions about light source,
spherical mirror, and observer. As we will see below, the interior problem
is equivalent to finding the maximal ellipse with foci at $z_1, z_2$
contained in the unit disk, and the point of reflection
$u \in \partial {\mathbb{D}}$ is the tangent point of the ellipse with
the circumference.
Algebraically, this leads to the solution of a quartic equation as we will
see below.

We met this problem in a different context, in the study of the triangular
ratio metric $s_G$ of a given domain $G \subset {\mathbb{R}}^2$ defined as
follows for $z_1,z_2 \in G$ \cite{hvz}
\begin{equation}  \label{eq:sdef}
  s_G(z_1,z_2)= \sup_{z \in \partial G} \frac{|z_1-z_2|}{|z_1-z|+|z-z_2|} \, .
\end{equation}
By compactness, this supremum is attained at some point $z_0 \in \partial
G\,.$ If $G$ is convex, it is simple to see that $z_0$ is the point of
contact of the boundary with an ellipse, with foci $z_1,z_2\,,$ contained in
$G\,.$ Now for the case $G= \mathbb{D}$ and $z_1,z_2 \in \mathbb{D}\,$, if the extremal point is $z_0
\in \partial \mathbb{D}\,,$ the connection between the triangular ratio
distance
\begin{equation*}
  s_{\mathbb{D}}(z_1, z_2) = \frac{|z_1- z_2|}{ |z_1- z_0| + |z_2-z_0| }
\end{equation*}
and the Ptolemy-Alhazen interior problem is clear: $u$=$z_0$ satisfies %
\eqref{eq:theptu}. Note that \eqref{eq:theptu} is just a reformulation of a basic
property of the ellipse with foci $z_1,z_2:$ the normal to the ellipse
(which in this case is the radius of the unit circle terminating at the
point $u$) bisects the angle formed by segments joining the foci $z_1,z_2$
with the point $u\,.$
During the past decade,
the $s_G$ metric has been studied in several
papers e.g. by P. H\"ast\"o \cite{h1,h2}; the interested reader is referred
to \cite{hvz} and the references there.

We study the Ptolemy-Alhazen interior problem and in our main result,
Theorem \ref{thm:Fuji}, we give an equation of degree four that yields the
reflection point on the unit circle. Standard symbolic computation software can then be used to
find this point numerically. We also study the Ptolemy-Alhazen exterior
problem. 

\begin{thm}\label{thm:Fuji}
  The point $u $ in \eqref{eq:theptu} is given as a
  solution of the  equation
  \begin{equation}  \label{eq:equation}
     \overline{z_1}\overline{z_2}u^4-(\overline{z_1}+\overline{z_2})u^3
     +(z_1+z_2)u-z_1z_2=0\,.
  \end{equation}
\end{thm}

It should be noticed that the equation \eqref{eq:equation} may have roots in
the complex plane that are not on the unit circle, and of the roots on the
unit circle, we must choose one root $u\,,$ that minimizes the sum $%
|z_1-u|+|z_2-u|\,.$ We call this root \textit{the minimizing root} of %
\eqref{eq:equation}.

\begin{cor}\label{cor:sDformula}
  For $z_1, z_2 \in \mathbb{D}$ we have
  \begin{equation*}
    s_{\mathbb{D}}(z_1, z_2) = \frac{|z_1- z_2|}{|z_1-u|+ |z_2-u| }
  \end{equation*}
  where $u \in \partial {\mathbb{D} } $ is the minimizing root of
  \eqref{eq:equation}\,.
\end{cor}

As we will see below, the minimizing root need not be unique.

We have used Risa/Asir symbolic computation software \cite{nst} in the
proofs of our results. We give a short Mathematica code for the computation
of $s_{\mathbb{D}}(z_1, z_2)\,.$


Theorem \ref{thm:Fuji} is applicable not merely to light signals but
whenever the angles of incidence and reflection of a wave or signal are
equal, for instance in the case of electromagnetic signals like radar
signals or acoustic waves. H. Bach \cite{ba} has made numerical studies of Alhazen's
ray-tracing problem related to circles and ellipses. A.R. Miller and E. Vegh
\cite{mv} have studied the Ptolemy-Alhazen problem in terms of quartic
equations. However, their quartic equation is not the same as %
\eqref{eq:equation}. Mathematical theory of billiards also leads to similar
studies: see for instance the paper by M. Drexler and M.J. Gander \cite{dg}.
The Ptolemy-Alhazen problem also occurs in computer graphics and catadioptric optics
 \cite{atr}. The well-known lithograph of M. C. Escher named \href{https://en.wikipedia.org/wiki/Hand_with_Reflecting_Sphere}{"\underline{Hand with reflecting sphere}"} demonstrates nicely the idea of catadioptric optics.

\section{Algebraic solution to the Ptolemy-Alhazen problem}

\label{section2}

In this section we prove Theorem \ref{thm:Fuji} and give an algorithm for computing $s_{\mathbb{D}}(z_1,z_2)$ for $z_1,z_2\in \mathbb{D}\,$.

\begin{problem}
  For $z_1,z_2\in\mathbb{D} $, find the point $u \in\partial\mathbb{D} $ such
  that the sum $|u-z_1|+|u-z_2| $ is minimal.
\end{problem}

The point $u $ is given as the point of tangency of an ellipse $%
|z-z_1|+|z-z_2|=r $ with the unit circle.


\begin{rem}\label{lem:myLemma1}
  For $z_1,z_2\in\mathbb{D} $, if $u\in\partial\mathbb{D} $ is the point of
  tangency of an ellipse $|z-z_1|+|z-z_2|=r $ and the unit circle, then $r $
  is given by
  \begin{equation*}
    r=|2-\overline{u}z_1-u\overline{z_2}|\,.
  \end{equation*}
  In fact, from the ``reflective property''
  $\measuredangle (z_{1},u,0)=\measuredangle (0,u,z_{2})$
  of an ellipse,  the following holds
  \begin{equation}\label{eq:hansha}
    \arg\dfrac{u}{u-z_1}=\arg\dfrac{u-z_2}{u}
     =-\arg\dfrac{\overline{u}-\overline{z_2}}{\overline{u}},
  \end{equation}
  and
  \begin{equation}\label{eq:arg}
    \arg(\overline{u}(u-z_1))=\arg(u(\overline{u}-\overline{z_2})).
  \end{equation}
  Since the point $ u $ is on the ellipse $|z-z_1|+|z-z_2|=r $ and
  satisfies $u\overline{u}=1 $, we have
  $$
      r = |u-z_1|+|u-z_2|
        = |\overline{u}(u-z_1)|+|u(\overline{u}-\overline{z_2})|
        = |\overline{u}(u-z_1)+u(\overline{u}-\overline{z_2})|
        =|2-\overline{u}z_1-u\overline{z_2}|\,.
  $$
\end{rem}


\noindent
\begin{nonsec}
  {\textit{Proof of Theorem \ref{thm:Fuji}.}} 
\end{nonsec}


%
From the equation \eqref{eq:hansha}, we have
  $$
    \arg \Big(\frac{u-z_1}{u}\cdot\frac{u-z_2}{u}\Big)=0.
  $$
This implies $ \dfrac{(u-z_1)(u-z_2)}{u^2} $ is real and its complex
conjugate is also real.
Hence,
$$
  \frac{(u-z_1)(u-z_2)}{u^2}
  =\frac{(\overline{u}-\overline{z_1})(\overline{u}-\overline{z_2})}
  {\overline{u^2}}
$$
holds.
Since $ u $ satisfies $ u\overline{u}=1 $, we have the assertion.
\hfill{$\square$}

\begin{rem}
  The solution of \eqref{eq:equation} includes all the tangent points of the
  ellipse $|z-z_1|+|z-z_2|=|2-\overline{u}z_1-u\overline{z_2}| $ and the unit
  circle. (See Figures \ref{fig:fig1}, \ref{fig:myfig2}.)
  . Figure \ref{fig:myfig2} displays a situation
  where all the roots of
  the quartic equation have unit modulus. However, this is not always the case
  for the equation  \eqref{eq:equation}. E.g. if
  $z_1=0.5+(0.1 \cdot k ) i, k=1,..,5, \, z_2= 0.5,$
  the equation \eqref{eq:equation}  has two roots of modulus
  equal to $1$ and two roots off the unit circle.  Miller and Vegh \cite{mv}
  have also studied the Ptolemy-Alhazen  problem using a quartic equation,
  that is different from our equation and, moreover, all the roots of their
  equation have modulus equal to one.

\end{rem}


We say that a polynomial $P(z)$ is \textit{self-inversive} if
$P(1/\overline{u})= 0$ whenever $u \neq 0$ and $P(u)=0\,.$
It is easily seen that the quartic polynomial in \eqref{eq:equation}
is self-inversive. Note that the points $u$ and  $1/\overline{u}$
are obtained from each other by the
inversion transformation $w \mapsto 1/\overline{w} \,.$

\begin{lem} \label{lem:mylem28}
  The equation \eqref{eq:equation} always has at least two
  roots of modulus equal to $1\,.$
\end{lem}

\begin{proof}
  Consider first the case, when $z_1 z_2 =0 \,. $ In this case the
  equation \eqref{eq:equation} has two roots $u, |u|=1,$ with
  $u^2={z_1}/{\overline{z_1}}\in\partial\mathbb{D} $
  if $z_2=0, z_1\neq 0\,. $ (The case $z_1=z_2=0 $ is trivial.)
  Suppose that the equation has no root on the unit circle
  $\partial\mathbb{D} \,.$

  By the invariance property pointed out above, if $u_0\in \mathbb{C}
  \setminus ( \{0\} \cup \partial\mathbb{D}) $ is a root of
  \eqref{eq:equation}\,, then ${1}/{\overline{u_0}} $ also is a root
  of \eqref{eq:equation}.
  Hence the number of roots off the unit circle is even and the number of
  roots on the unit circle must also be even. We will now show that  this even
  number is either $2$ or $4\,.$

  Let $a,b,\alpha,\beta\in\mathbb{R}, 0<a<1,0<b<1\,,$ and let
  \begin{equation*}
    ae^{i\alpha},\ \frac{1}{a}e^{i\alpha},  \ be^{i\beta},\
   \frac{1}{b}e^{i\beta}
  \end{equation*}
  be the four roots of the equation \eqref{eq:equation}\,.
  Then, the equation
  \begin{equation}  \label{eq:1}
    z_1z_2(u-ae^{i\alpha})(u-\frac{1}{a}e^{i\alpha})
          (u-be^{i\beta})(u-\frac{1}{b}e^{i\beta})=0
  \end{equation}
  coincides with \eqref{eq:equation}.
  Therefore, the coefficient of degree 2 of \eqref{eq:1} vanishes,
  and we have
  \begin{equation}  \label{eq:zero}
    e^{i2\alpha}+e^{i2\beta}
      =-\Big(a+\frac1a\Big)\Big(b+\frac1b\Big)e^{i(\alpha+\beta)}.
  \end{equation}
  The absolute value of the left hand side of \eqref{eq:zero} satisfies
  \begin{equation}  \label{eq:left}
    |e^{i2\alpha}+e^{i2\beta}|\leq 2\,.
  \end{equation}
  On the other hand, the absolute value of the right hand side of %
  \eqref{eq:zero} satisfies
  \begin{equation}  \label{eq:right}
    \Big|(a+\frac1a)(b+\frac1b)e^{i2(\alpha+\beta)}\Big|
      =\Big|a+\frac1a\Big|\Big|b+\frac1b\Big|>4\,,
  \end{equation}
  because the function $f(x)=x+\frac1x $ is monotonically decreasing on
  $0 <x\leq 1 $ and $f(1)=2 $. The inequalities \eqref{eq:left} and %
  \eqref{eq:right} imply that the equality \eqref{eq:zero} never holds.
  Hence \eqref{eq:equation} has roots of modulus equals to $1.\,$
\end{proof}   

\begin{figure}[htbp]
  \centerline{\includegraphics[width=0.4\linewidth]{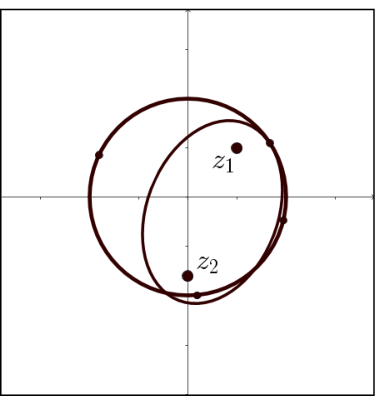}\
              \includegraphics[width=0.4\linewidth]{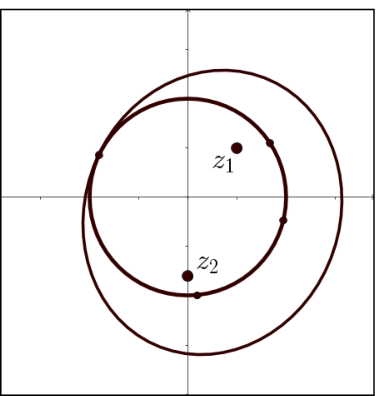}
             }
  \centerline{\includegraphics[width=0.4\linewidth]{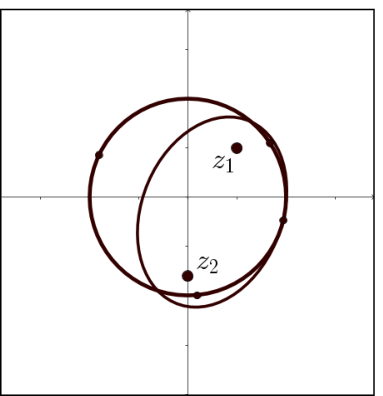}\
              \includegraphics[width=0.4\linewidth]{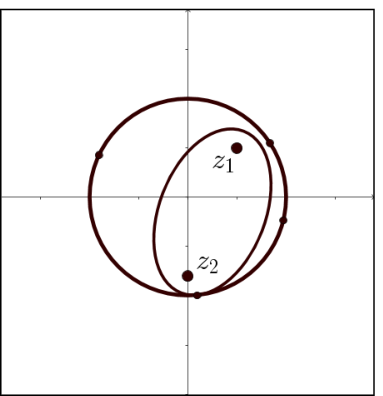}
             }
  \caption{This figure indicates the four solutions of \eqref{eq:equation}
           (dots on the unit circle) and the ellipse that corresponds
           to each $u $, for $ z_1=0.5+0.5i,z_2=-0.8i $.
           The figure on the lower right shows the point $u $
           that gives the minimum. }
  \label{fig:myfig2}
\end{figure}


\begin{figure}[htbp]
  \centerline{\includegraphics[width=0.4\linewidth]{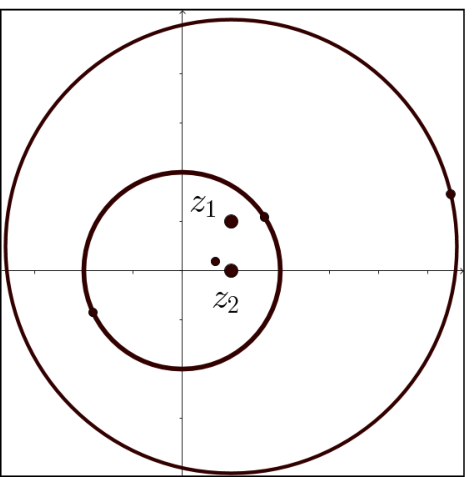}\
              \includegraphics[width=0.4\linewidth]{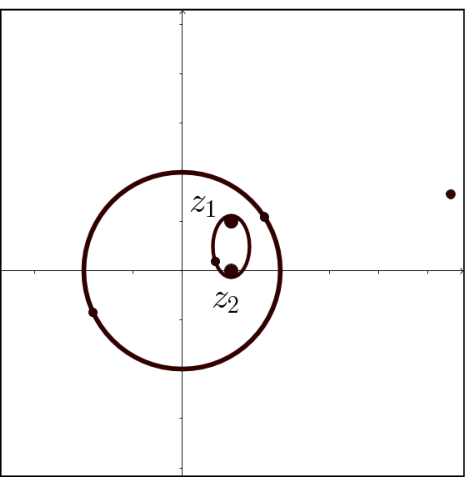}
             }
  \centerline{\includegraphics[width=0.4\linewidth]{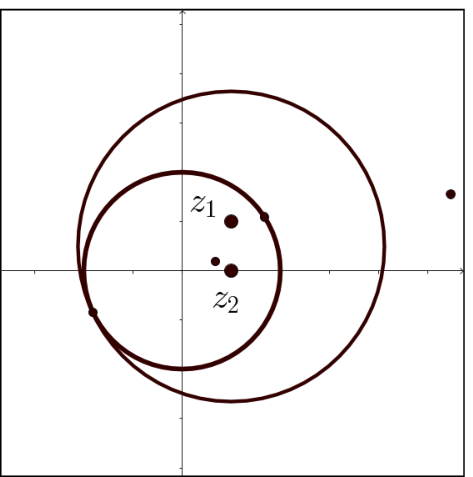}\
              \includegraphics[width=0.4\linewidth]{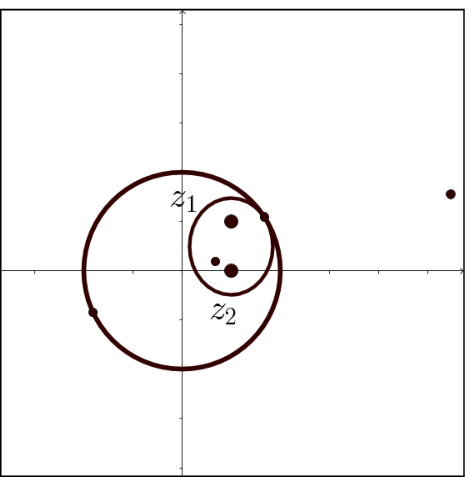}
             }
  \caption{For $z_1=0.5+0.5i$ and $z_2= 0.5$, there are only two solutions of
           \eqref{eq:equation} on the unit circle.
           The figure on the lower right shows the point $u $
           that gives the minimum.}
  \label{fig:fig3}
\end{figure}


\begin{rem}
  We consider here several special cases of the equation
  \eqref{eq:equation} and for some special cases we give the corresponding
  formula for the $s_{\mathbb{D}}$ metric which readily follows from
  Corollary \ref{cor:sDformula}.
  \begin{enumerate}[{\bf {Case} 1.}]
    \item \label{item:case1}
           $z_{1}\neq 0=z_{2}$ (cubic equation). The equation
           \eqref{eq:equation} is now
           $\left( -\overline{z_{1}}\right) u^{3}+z_{1}u=0$
           and has the roots $u_{1}=0$,
           $u_{2,3}=\pm \frac{z_{1}}{\left\vert z_{1}\right\vert }$ and
           for $z \in {\mathbb{D}}$
           \begin{equation*}
            s_{\mathbb{D}}(0,z) = \frac{|z|}{2 - |z|}\,.
           \end{equation*}
    \item \label{item:case2}
           $z_{1}+z_{2}=0$, $z_{1}\neq 0$. The equation
           \eqref{eq:equation} reduces now to:
           \begin{equation*}
             \left( -\overline{z_{1}}^{2}\right) u^{4}+z_{1}^{2}
              =0\Leftrightarrow u^{4}
              =\left( \frac{z_{1}}{\overline{z_{1}}}\right) ^{2}\Leftrightarrow u^{4}
              =\left( \frac{z_{1}}{\left\vert z_{1}\right\vert }\right) ^{4}\,.
           \end{equation*}

           \noindent The roots are:
           $u_{1,2}=\pm \frac{z_{1}}{\left\vert z_{1}\right\vert }$,
           $u_{3,4}=\pm i\frac{z_{1}}{\left\vert z_{1}\right\vert }$
           (four distinct roots of modulus $1$) and for $z \in {\mathbb{D}}$
           \begin{equation*}
              s_{\mathbb{D}}(z,-z) = |z|\,.
           \end{equation*}
    \item \label{item:case3}
           $z_{1}=z_{2}\neq 0\,.$ Clearly $s_{\mathbb{D}}(z,z)=0\,.$
           Denote $z:=z_{1}=z_{2}$. The equation \eqref{eq:equation}
           reduces now to:
           \begin{equation*}
             \overline{z}^{2}u^{4}-2\overline{z}u^{3}+2zu-z^{2}
               = (\overline{z} u^2 -z)(\overline{z} u^2-2u +z)= 0\,.
           \end{equation*}

           \noindent 
           Then we see that $u_{1,2}=\pm \frac{z}{\left\vert z\right\vert }$
           are roots.

\noindent

          \noindent The other roots are:
          \begin{enumerate}[1)]
             \item  If $\left\vert z\right\vert <1$,
                   then $u_{3,4}=\frac{1}{\overline{z}}%
                   \left( 1\pm \sqrt{1-\left\vert z\right\vert ^{2}}\right) $
                   (with $\left\vert u_{3}\right\vert >1$,
                    $\left\vert u_{4}\right\vert <1$)
             \item  If $\left\vert z\right\vert >1$,
                    then $u_{3,4}=\frac{1}{\overline{z}}%
                   \left( 1\pm i\sqrt{\left\vert z\right\vert ^{2}-1}\right) $
                   (with $\left\vert u_{3}\right\vert =\left\vert u_{4}\right\vert =1$).
          \end{enumerate}
    \item \label{item:case4}
          $\left\vert z_{1}\right\vert =\left\vert z_{2}\right\vert\neq 0\,.$

          \noindent Denote $\rho =\left\vert z_{1}\right\vert =\left\vert
          z_{2}\right\vert $. Using a rotation around the origin and a change of
          orientation we may assume that $\arg z_{2}=-\arg z_{1}=:\alpha $,
          where $0\leq \alpha \leq \frac{\pi }{2}$\,.

          \noindent The equation \eqref{eq:equation} reads now:
          $\rho ^{2}u^{4}-2\rho\left( \cos \alpha \right) u^{3}+2\rho
          \left( \cos \alpha \right) u-\rho^{2}=0$

          $\rho ^{2}u^{4}-2\rho \left( \cos \alpha \right) u^{3}
           +2\rho \left( \cos\alpha \right) u-\rho ^{2}
           =\rho ^{2}\left( u^{2}-1\right) \left( u^{2}-\frac{%
           2\cos \alpha }{\rho }u+1\right) $

          \noindent The roots are: $u_{1,2}=\pm 1$ and
          \begin{enumerate}[1)]
             \item  If $0<\rho <\cos \alpha $,
                    then $u_{3,4}=\frac{\cos \alpha }{\rho }\pm
                    \sqrt{\left( \frac{\cos \alpha }{\rho }\right) ^{2}-1}$
                   (here $\left\vert u_{3}\right\vert >1$,
                   $\left\vert u_{4}\right\vert <1$)
             \item If $\rho \geq \cos \alpha $,
                   then $u_{3,4}=\frac{\cos \alpha }{\rho }\pm i
                   \sqrt{1-\left( \frac{\cos \alpha }{\rho }\right) ^{2}}$
                   (here $\left\vert u_{3}\right\vert =\left\vert u_{4}\right\vert =1$).
          \end{enumerate}
          \noindent Note that Case \ref{item:case4} includes Cases
          \ref{item:case2} and \ref{item:case3}
          (for $\alpha =\frac{\pi }{2}$, respectively $\alpha =0$)\,.

    \item \label{item:case5}
          $z_{1}=tz_{2}$ ($t\in \mathbb{R}$, $z_{2}\neq 0$). This
          case  is generalization of cases $z_{1}=0\neq z_{2}$,
          $z_{1}+z_{2}=0$, $z_{1}\neq 0$ and $z_{1}=z_{2}\neq 0$.

          Denote $P\left( u\right)
          =\overline{z_{1}z_{2}}u^{4}-\left( \overline{z_{1}}+
          \overline{z_{2}}\right) u^{3}+\left( z_{1}+z_{2}\right) u-z_{1}z_{2}$.

          Denoting $z_{2}=z$ we have:

         \begin{align*}
          P(u)
            =& t\overline{z}^{2}u^{4}-\left( 1+t\right) \overline{z}
              u^{3}+\left(1+t\right) zu-tz^{2} \\
           = & t\overline{z}^{2}\left( u^{4}
               -\frac{z^{4}}{\left\vert z\right\vert ^{4}}\right)
               - \left( 1+t\right) \overline{z}u\left( u^{2}
               -\frac{z^{2}}{\left\vert z\right\vert ^{2}}\right) \,.
        \end{align*}
        \begin{equation*}
           P(u)=\overline{z}\left( u-\frac{z}{\left\vert z\right\vert }\right)
              \left( u+\frac{z}{\left\vert z\right\vert }\right)
              \left( t\overline{z}u^{2}-\left(1+t\right) u+tz\right)
        \end{equation*}

        For $t=0$ the roots of $P$ are $0,\pm \frac{z}{\left\vert z\right\vert }$.

        Let $t\neq 0$. Besides $\pm \frac{z}{\left\vert z\right\vert }$ there are
        two roots, which have modulus $1$ if and only if $\left\vert z\right\vert
        \geq \left\vert \frac{1+t}{2t}\right\vert $\,.
  \end{enumerate}
\end{rem}


\begin{nonsec}
  \textbf{Exterior Problem.}
  Given $z_1,z_2\in\mathbb{C}\setminus\overline{\mathbb{D}} $,
  find the point $u \in \partial \mathbb{D} $ such that the sum
  $|z_1-u|+|u-z_2| $ is minimal.
\end{nonsec}


\begin{lem}
  If the segment $[z_1,z_2 ]$ does not intersect with $\partial\mathbb{D} $,
  the point $u $ is given as a solution of the equation
  \begin{equation*}
    \overline{z_1}\overline{z_2}u^4-(\overline{z_1}+\overline{z_2})u^3
    +(z_1+z_2)u-z_1z_2=0\,.
  \end{equation*}
\end{lem}

\begin{rem}
  The above equation coincides with the equation \eqref{eq:equation} for the
  ``interior problem'', since Theorem \ref{thm:Fuji} could be proved
  without using the assumption $ z_1,z_2\in\mathbb{D} $.
\end{rem}

\begin{rem}
  The equation of the line joining two points $z_1 $ and $z_2 $ is given by
  \begin{equation}  \label{eq:line}
    \frac{z_1-z}{z_2-z}
       =\frac{\overline{z_1}-\overline{z}}{\overline{z_2}-\overline{z}}\,.
  \end{equation}
  Then, the distance from the origin to this line is
  \begin{equation*}
    \frac{|\overline{z_1}z_2-z_1\overline{z_2}|}{2|z_1-z_2|}.
  \end{equation*}
  Therefore, if two points $z_1,z_2 $ satisfy
  $\dfrac{|\overline{z_1}z_2-z_1\overline{z_2}|}{2|z_1-z_2|}\leq 1$\,,
  the line \eqref{eq:line} intersects
  with the unit circle, and the triangular ratio metric
  $s_{\mathbb{D}}(z_1,z_2)=1 $\,.
\end{rem}

\begin{lem} \label{lem:algcur}
  The boundary of
  $B_s(z,t) =\{ w \in \mathbb{D}: s_{\mathbb{D}}(z,w)<t \} $
  is included in an algebraic curve.
\end{lem}

\begin{proof}
  Without loss of generality, we may assume that the center point $z=:c $
  is on the positive real axis. Then,
  \begin{align}  \label{eq:smet}
    s_{\mathbb{D}}(c,w)
      &=\sup_{\zeta\in\partial\mathbb{D}}
       \frac{|c-w|}{|c-\zeta|+|\zeta-w|}  \notag \\
      &=\frac{|c-w|}{|2-\overline{u}c-u\overline{w}|}
       \quad (\mbox{from {Remark} \ref{lem:myLemma1}}),
  \end{align}
  where $u $ is a minimizing root of the equation
  \begin{equation}  \label{eq:U}
    U_c(w)={c\overline{w}}u^4-(c+\overline{w})u^3+(c+w)u-cw=0\,.
  \end{equation}

  Moreover, $B_s(0,t)=\{|w|<\frac{2t}{1+t}\} $ (resp. $B_s(c,0)=\{c\} $)
  holds for $c=0 $ (resp. $t=0$), and $B_s(c,t)=\{0\} $ holds if and only
  if $c=0 $ and $t=0 $. Therefore we may assume that $c\neq0 $, $t\neq0 $
  and $w\not\equiv0 $\,.

  Now, consider the following system of equations $s_{\mathbb{D}}(c,w)=t $
  and $U_c(w)=0 $, i.e,
  \begin{equation}  \label{eq:SU}
    S_{c,t}(w)=t^2|2-\overline{u}c-u\overline{w}|^2-|c-w|^2=0
      \quad \mbox{and} \quad U_c(w)=0\,.
  \end{equation}
  The above two equations have a common root if and only if both of
  polynomials $S_{c,t}(w) $ and $U_c(w) $ have non-zero leading coefficient
  with respect to $u $ variable and the resultant satisfies
  $\mbox{resultant}_u(S_{c,t},U_c) = 0 $. Using the ``resultant''
  command of the Risa/Asir software, we have
  \begin{equation*}
    \mbox{resultant}_u(S_{c,t},U_c)= cw\overline{w}\cdot
     {\mathcal{B}}_{c,t}(w)\,,
  \end{equation*}
  where 
  \begin{align*}
    & {\mathcal{B}}_{c,t}(w) \\
    &=(\overline{w}c-1)(wc-1)\big((c^2+w\overline{w}-2)^2 -4(\overline{w}%
     c-1)(wc-1)\big)^2t^8 \\
    &-(c-w)(c-\overline{w})\big(4\overline{w}wc^8-3(w+\overline{w})c^7 -2(2%
     \overline{w}^2w^2+2\overline{w}w-1)c^6 \\
    &-(w+\overline{w})(13w\overline{w}+2)c^5-2(2\overline{w}^3w^3 -(36\overline{w%
     }^2+10)w^2-27\overline{w}w \\
    & -10\overline{w}^2-4)c^4-(w+\overline{w})(13\overline{w}^2w^2 +92\overline{w%
     }w+32)c^3 \\
    & +2(w\overline{w}(2\overline{w}^3w^3-2\overline{w}^2w^2 +27\overline{w}%
     w+48)+2(5w\overline{w}+2)(w^2+\overline{w}^2))c^2 \\
    & -w\overline{w}(w+\overline{w})(3\overline{w}^2w^2+2\overline{w}w+32)c +2w^2%
     \overline{w}^2(w\overline{w}+4)\big)t^6 \\
    &+(c-w)^2(c-\overline{w})^2\big(6\overline{w}wc^6-3(w+\overline{w})c^5 +(4%
     \overline{w}^2w^2+16\overline{w}w+1)c^4 \\
    &-2(w+\overline{w})(13w\overline{w}+5)c^3+(6\overline{w}^3w^3 +(16\overline{w%
     }^2+1)w^2+52\overline{w}w+\overline{w}^2)c^2 \\
    &-w\overline{w}(w+\overline{w})(3w\overline{w}+10)c +\overline{w}^2w^2\big)%
     t^4 \\
    &-c(c-w)^3(c-\overline{w})^3\big(4w\overline{w}c(c^2+w\overline{w}+3) -(c^2+w%
     \overline{w})(w+\overline{w})\big)t^2 \\
    &+c^2w\overline{w}(c-w)^4(c-\overline{w})^4\,.
  \end{align*}
  Moreover, we can check that
  \begin{equation*}
     \mathcal{B}_{c,0}(w)=|w|^2c^2|c-w|^8
  \end{equation*}
  and
  \begin{equation*}
    \mathcal{B}_{0,t}(w)=|w|^4t^4  \big((t-1)^2|w|^2-4t^2\big)
    \big((t+1)^2|w|^2-4t^2\big)\,.
  \end{equation*}
  Hence, the boundary of $B_s(c,t) $ is included in the algebraic curve
  defined by the equation ${\mathcal{B}}_{c,t}(w)=0 $\,.
\end{proof}

\begin{rem}
  The algebraic curve $\{ w: {\mathcal{B}}(w)=0 \}$ does not coincide with the boundary
  $\partial B_s(c,t) $. There is an ``extra'' part of the curve since the
  equation \eqref{eq:U} contains extraneous solutions.
\end{rem}

\begin{figure}[tbp]
  \includegraphics[width=0.5\linewidth]{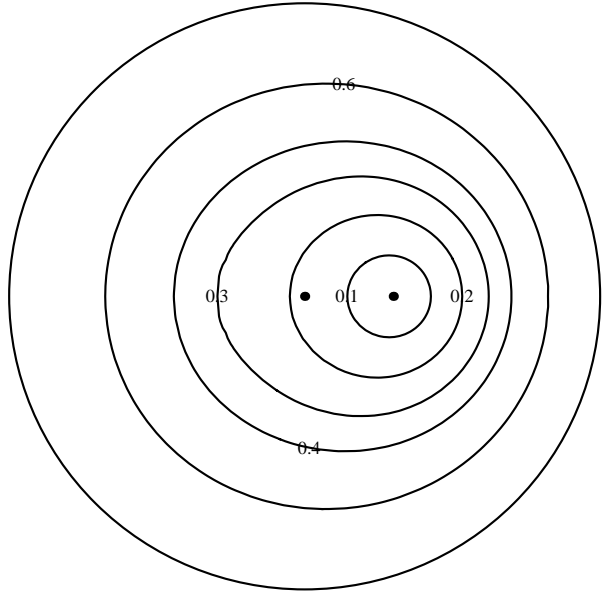}
  \caption{Level sets $\{x+ i y: s_{\mathbb{D}}(0.3, x+i y) =t \}$ for
           $t=0.1, 0.2, 0.3, 0.4, 0.6\,$ and the unit circle.
           By Lemma \ref{lem:algcur} these level sets are
           contained in an algebraic curve.}
  \label{fig:sDcircles}
\end{figure}


The analytic formula in Corollary \ref{cor:sDformula} for the triangular
ratio metric $s_{\mathbb{D}}(z_1,z_2)$ is not very practical. Therefore we
next give an algorithm based on Theorem \ref{thm:Fuji} for the evaluation
of the numerical values.

\medskip \textbf{Algorithm.} We next give a Mathematica algorithm for
computing $s_{\mathbb{D}}(x,y)$ for given points $x,y \in{\mathbb{D}}\,.$

\small

\begin{verbatim}

   sD[x_, y_] := Module[{u, sol, mySol, tmp = 2*Sqrt[2]},
   sol = Solve[ Conjugate[ x*y] u^4 -  Conjugate[x + y] u^3 +
      (x + y) u - x*y == 0, {u}];
   mySol = u /. sol;
   Do[If[Abs[Abs[mySol[[i]] ] - 1] < 10^(-12),
      tmp = Min[tmp,
        Abs[mySol[[i]] - x] + Abs[mySol[[i]] - y]]], {i, 1, Length[mySol]}];
    Abs[x - y]/tmp] ;


\end{verbatim}

 \normalsize

{
}








\section{Geometric approach to the Ptolemy-Alhazen problem}

{ 
}

In this section the unimodular roots of equation %
\eqref{eq:equation} are characterized as points of intersection of a conic
section and the unit circle, then $n$ such roots are studied, where $n=4$
in the case of the exterior problem and $n=2 $ in the case of the interior
problem. We describe the construction of the conic section mentioned above. Except in the cases where $0,z_{1},z_{2}$ are collinear or $
\left\vert z_{1}\right\vert =\left\vert z_{2}\right\vert \,, $
the construction cannot be carried out
as ruler-and-compass construction.
 Neumann
\cite{neu} proved that Alhazen's interior problem for points $z_{1},z_{2}$
is solvable by ruler and compass only for $(\mathrm{Re}z_{1},\mathrm{Im}%
z_{1},\mathrm{Re}z_{2},\mathrm{Im}z_{2})$ belonging to a null subset of $%
\mathbb{R}^{4}$, in the sense of Lebesgue measure.

We characterize algebraically condition \eqref{eq:theptu} without
assuming that $z_{1},z_{2}\in \mathbb{D}$, or $z_{1},z_{2}\in \mathbb{C}%
\setminus \overline{\mathbb{D}}$, or $u\in \partial \mathbb{D}$\,.

\begin{lem}
  Let $z_{1},z_{2}\in \mathbb{C}$ and $u\in \mathbb{C}^{\ast }%
  \mathbb{\setminus }\left\{ z_{k}:k=1,2\right\} $. The following are
  equivalent:
  \begin{enumerate}[{\rm (i)}]
    \item \label{item:i}
          $\measuredangle (z_{1},u,0)=\measuredangle (0,u,z_{2})$.
    \item \label{item:ii}
          $\frac{u^{2}}{\left( u-z_{1}\right) \left( u-z_{2}\right) }
            =\frac{\overline{u}^{2}}{\left( \overline{u}-\overline{z_{1}}\right)
               \left(\overline{u}-\overline{z_{2}}\right) }$ and
          $\frac{u^{2}}{\left(u-z_{1}\right) \left( u-z_{2}\right) }
            +\frac{\overline{u}^{2}}{\left(\overline{u}-\overline{z_{1}}\right)
           \left( \overline{u}-\overline{z_{2}}\right) }>0$;
    \item \label{item:iii}
        \begin{equation}\label{eq:cubic1}
          \overline{z_{1}z_{2}}u^{2}-\left(\overline{z_{1}}+\overline{z_{2}}\right)
          \overline{u}u^{2}+\left( z_{1}+z_{2}\right) \overline{u}^{2}u-z_{1}z_{2}%
          \overline{u}^{2}=0\
        \end{equation}
         and
        \begin{equation} \label{eq:ineq1}
          \overline{z_{1}z_{2}}u^{2}-\left( \overline{z_{1}}+\overline{z_{2}}\right)
          \overline{u}u^{2}-\left( z_{1}+z_{2}\right) \overline{u}^{2}u+z_{1}z_{2}%
          \overline{u}^{2}+2u^{2}\overline{u}^{2}>0\,.
        \end{equation}
  \end{enumerate}
\end{lem}

\begin{proof}
  Let $u\in \mathbb{C}^{\ast }\mathbb{\setminus }\left\{
  z_{k}:k=1,2\right\} $.
  Clearly, $ \measuredangle (z_{1},u,0) =\arg\frac{u}{u-z_{1}}$
  and $\measuredangle (0,u,z_{2}) =\arg \frac{u-z_{2}}{u}$\,.
  Denoting $v:=\frac{u}{u-z_{1}}:\frac{u-z_{2}}{u}$\,, we see that
  $\measuredangle (z_{1},u,0) =\measuredangle (0,u,z_{2})$ if and only if
  $v$ satisfies both $v=\overline{v}$ and $v+\overline{v}>0$\,, i.e.
  if and only if (\ref{item:ii}) holds.

   We have $v=\overline{v}$ (respectively, $v+\overline{v}>0$) if
  and only if \eqref{eq:cubic1} (respectively, \eqref{eq:ineq1}) holds,
  therefore (\ref{item:ii}) and (\ref{item:iii}) are equivalent.

   In the special case $z_{1}=z_{2}=0$ ($z_{1}=z_{2}\neq 0$) (\ref{item:i}),
  (\ref{item:ii}) and (\ref{item:iii}) are satisfied whenever
  $u\in \mathbb{C}^{\ast }$ (respectively, if and only if
  $u=\lambda z_{1}$ for some real number $\lambda \neq 0,1$)\,.
\end{proof}

\begin{rem}
  Let $u\in \mathbb{C}^{\ast }\mathbb{\setminus }\left\{z_{k}:k=1,2\right\}\,. $
  If
  \begin{equation*}
    \frac{u^{2}}{\left( u-z_{1}\right) \left( u-z_{2}\right) }
      =\frac{\overline{u}^{2}}{\left( \overline{u}-\overline{z_{1}}\right)
       \left( \overline{u}-\overline{z_{2}}\right) }
    \,\,\mathrm{and} \, \,
    \frac{u^{2}}{\left(u-z_{1}\right) \left( u-z_{2}\right) }
       +\frac{\overline{u}^{2}}{\left(\overline{u}-\overline{z_{1}}\right)
        \left( \overline{u}-\overline{z_{2}}\right) }<0\,,
  \end{equation*}
  then $\left\vert \measuredangle (z_{1},u,0)
  -\measuredangle (0,u,z_{2})\right\vert =\pi $\,. The converse also holds.

  Consider the interior problem, with $z_{1},z_{2}\in \mathbb{D}$
  and $u\in \partial \mathbb{D}$. The unit circle is exterior to the circles
  of diameters $\left[ 0,z_{1}\right] $, $\left[ 0,z_{2}\right] $. An elementary
  geometric argument shows that
  $-\frac{\pi }{2}< \measuredangle(z_{1},u,0)<\frac{\pi }{2}$ and
  $-\frac{\pi }{2}<\measuredangle(0,u,z_{2}) <\frac{\pi }{2}$\,,
  therefore $\left\vert \measuredangle (z_{1},u,0)
   -\measuredangle (0,u,z_{2})\right\vert \neq \pi $\,.
  In this case \eqref{eq:cubic1} implies
  $\measuredangle (z_{1},u,0) =\measuredangle (0,u,z_{2})$\,.
\end{rem}

  The equation \eqref{eq:cubic1} defines a curve passing through $0$,
  $z_{1}$ and $z_{2}$, that is a cubic if $z_{1}+z_{2}\neq 0$,
  respectively a conic section if $z_{1}+z_{2}=0$ with
  $z_{1},z_{2}\in \mathbb{C}^{\ast }$\,.
  Then  under the inversion with respect to the unit circle, the image of the curve
  given by \eqref{eq:cubic1}  has the equation
  \begin{equation}\label{eq:conic1}
    \overline{z_{1}z_{2}}u^{2}-\left( \overline{z_{1}}+\overline{z_{2}}\right)
    u+\left( z_{1}+z_{2}\right) \overline{u}-z_{1}z_{2}\overline{u}^{2}=0\,.
  \end{equation}
  This is a conic section, that degenerates to a line if $z_{1}z_{2}=0$
  with $z_{1},z_{2}$ not both zero.

\begin{rem}
   If $u\in \partial \mathbb{D}$, then \eqref{eq:cubic1}
  (respectively, \eqref{eq:conic1}) holds if and only if
  \begin{equation*}
    \overline{z_{1}z_{2}}u^{2}-\left( \overline{z_{1}}+\overline{z_{2}}\right)
    u+\left( z_{1}+z_{2}\right) \frac{1}{u}-z_{1}z_{2}\frac{1}{u^{2}}=0\,.
  \end{equation*}
  The equations \eqref{eq:conic1}, \eqref{eq:cubic1} and \eqref{eq:equation} have
  the same unimodular roots.
\end{rem}

{ 
}

\begin{lem}
  Let $z_{1},z_{2}\in \mathbb{C}^{\ast }$. The conic section $\Gamma $
  given by \eqref{eq:conic1} has the center
  $c=\frac{1}{2}\left( \frac{1}{\overline{z_{1}}}+\frac{1}{\overline{z_{2}}}\right) $
  and it passes through $0,\frac{1}{\overline{z_{1}}}$, $\frac{1}{\overline{z_{2}}}$,
  $\frac{1}{\overline{z_{1}}}+\frac{1}{\overline{z_{2}}}$.
  If $\left\vert z_{1}\right\vert =\left\vert z_{2}\right\vert $
  or $\left\vert \arg z_{1}-\arg z_{2}\right\vert \in \left\{ 0,\pi \right\} $,
  then $\Gamma $ consists of the parallels $d_{1}$, $d_{2}$ through $c$
  to the bisectors (interior, respectively exterior) of the angle
  $\measuredangle (z_{1},0,z_{2})$\,.
  In the other cases $\Gamma $ is an equilateral hyperbola having the
  asymptotes $d_{1}$ and $d_{2}$\,.
\end{lem}

\begin{proof}
  The equation \eqref{eq:conic1} is equivalent to
  \begin{equation} \label{eq:conic2}
    \mathrm{Im}\left( \overline{z_{1}z_{2}}u
      \left( \frac{1}{\overline{z_{1}}}+\frac{1}{\overline{z_{2}}}-u\right)
      \right) =0\,.
  \end{equation}%
  The curve $\Gamma $ passes through the points $0$ and
  $2c=\frac{1}{\overline{z_{1}}}+\frac{1}{\overline{z_{2}}}$.
  If $u$ satisfies \eqref{eq:conic2}, then $2c-u$ also satisfies
  \eqref{eq:conic2}, therefore $\Gamma $ has the center $c$\,.
  Since $z_{1}$ and $z_{2}$ are on the cubic curve given by \eqref{eq:cubic1},
  $\Gamma $ passes through $\frac{1}{\overline{z_{1}}}$ and
  $\frac{1}{\overline{z_{2}}}$.
  The conic section $\Gamma $ is a pair of lines if and only if $\Gamma $
  passes through its center.
  For $u=\frac{1}{2}\left( \frac{1}{\overline{z_{1}}}+\frac{1}{\overline{z_{2}}}\right) $
  we have
  \begin{equation*}
    \mathrm{Im}\left( \overline{z_{1}z_{2}}u\left( \frac{1}{\overline{z_{1}}}+%
    \frac{1}{\overline{z_{2}}}-u\right) \right) =\frac{1}{4}\mathrm{Im}\left(
    \frac{\overline{z_{1}}}{\overline{z_{2}}}+\frac{\overline{z_{2}}}{\overline{%
    z_{1}}}\right) \,,
  \end{equation*}
  therefore $\Gamma $ is a pair of lines if and only if
  $\frac{\overline{z_{1}}}{\overline{z_{2}}}+\frac{\overline{z_{2}}}
    {\overline{z_{1}}}\in \mathbb{R}$.
  The following conditions are equivalent:

  (1) $\frac{\overline{z_{1}}}{\overline{z_{2}}}
       +\frac{\overline{z_{2}}}{\overline{z_{1}}}\in \mathbb{R}$; \
  (2) $\frac{z_{2}}{z_{1}}\in \mathbb{R}$ or
      $\left\vert \frac{z_{2}}{z_{1}}\right\vert =1$;  \
  (3) $\left\vert \arg z_{1}-\arg z_{2}\right\vert \in\left\{ 0,\pi \right\} $ or
      $\left\vert z_{1}\right\vert =\left\vert z_{2}\right\vert $.

  Denote $u=x+iy$. Using a rotation around the origin and a
  reflection we may assume that $\arg z_{2}=-\arg z_{1}=:\alpha $, where
  $0\leq \alpha \leq \frac{\pi }{2}$.
  In this case the equation of $\Gamma $ is
  \begin{equation}\label{eq:conic3}
    \left( x-\frac{\left\vert z_{1}\right\vert +\left\vert z_{2}\right\vert }{%
    2\left\vert z_{1}z_{2}\right\vert }\cos \alpha \right) \left( y-\frac{%
    \left\vert z_{2}\right\vert -\left\vert z_{1}\right\vert }{2\left\vert
    z_{1}z_{2}\right\vert }\sin \alpha \right) =\frac{\left\vert
    z_{2}\right\vert ^{2}-\left\vert z_{1}\right\vert ^{2}}{8\left\vert
    z_{1}z_{2}\right\vert ^{2}}\sin 2\alpha\,.
  \end{equation}

   The equation \eqref{eq:conic3} shows that $\Gamma $ is the pair of
  lines $d_{1}$, $d_{2}$ if $\left\vert z_{1}\right\vert =\left\vert
  z_{2}\right\vert $ or $\sin 2\alpha =0$, otherwise $\Gamma $ is an
  equilateral hyperbola having the asymptotes $d_{1}$ and $d_{2}$\,.
\end{proof}

\begin{lem}[Sylvester's theorem] \label{lem:syl}
  In any triangle with vertices $z_{1},z_{2},z_{3}$
  the orthocenter $z_{H}$ and the circumcenter $z_{C}$ satisfy the identity
  $z_{H}+2z_{C}=z_{1}+z_{2}+z_{3}$\,.
\end{lem}

\begin{proof}
  Let $z_{G}$ be the centroid of the triangle. It is well-known
  that $z_{G}=\frac{z_{1}+z_{2}+z_{3}}{3}$. By Euler's straightline theorem,
  $z_{H}-z_{G}=2(z_{G}-z_{C})$. Then $z_{H}+2z_{C}=3z_{G}=z_{1}+z_{2}+z_{3}$\,.
\end{proof}

\begin{lem} \label{lem:mylem511}
  Let $z_{1},z_{2}\in \mathbb{C}^{\ast }$. The orthocenter of the
  triangle with vertices $0,\frac{1}{\overline{z_{1}}}$,
  $\frac{1}{\overline{z_{2}}}$ belongs to the conic section
  given by equation \eqref{eq:conic1}.
\end{lem}

\begin{proof}
  Consider a triangle with vertices $z_{1},z_{2},z_{3}$ and
  denote by $z_{H}$ and $z_{C}$ the orthocenter and the circumcenter,
  respectively. By Sylvester's theorem, Lemma \ref{lem:syl},
  $z_{H}=z_{1}+z_{2}+z_{3}-2z_{C}$\,.

  But
  \begin{equation*}
     z_{C}=\det \left(
           \begin{array}{ccc}
              1 & 1 & 1 \\
              z_{1} & z_{2} & z_{3} \\
              \left\vert z_{1}\right\vert ^{2} & \left\vert z_{2}\right\vert ^{2} &
              \left\vert z_{3}\right\vert ^{2}%
           \end{array}%
          \right) :\det \left(
           \begin{array}{ccc}
              1 & 1 & 1 \\
              z_{1} & z_{2} & z_{3} \\
              \overline{z_{1}} & \overline{z_{2}} & \overline{z_{3}}%
           \end{array}%
          \right) \,.
  \end{equation*}
  If $z_{3}=0$, then
  $z_{C}=\frac{z_{1}z_{2}\left( \overline{z_{2}}
    -\overline{z_{1}}\right) }{z_{1}\overline{z_{2}}-\overline{z_{1}}z_{2}}$,
  hence
  \begin{equation*}
    z_{H}=\frac{\left( z_{1}-z_{2}\right) \left( z_{1}\overline{z_{2}}+\overline{%
    z_{1}}z_{2}\right) }{z_{1}\overline{z_{2}}-\overline{z_{1}}z_{2}}\,.
  \end{equation*}%
  Let $h$ be the orthocenter of the triangle with vertices
  $0,\frac{1}{\overline{z_{1}}}$,$\frac{1}{\overline{z_{2}}}$\,.
  The above formula implies
  \begin{equation}\label{eq:orthocenter}
    h=\frac{\overline{z_{2}}-\overline{z_{1}}}{\overline{z_{1}}\overline{z_{2}}}%
    \frac{z_{1}\overline{z_{2}}+\overline{z_{1}}z_{2}}{z_{1}\overline{z_{2}}-%
    \overline{z_{1}}z_{2}}\,.
  \end{equation}%
  Let $f( u ) :=$ $\overline{z_{1}}\overline{z_{2}}u^{2}-\left(
  \overline{z_{1}}+\overline{z_{2}}\right) u+\left( z_{1}+z_{2}\right)
  \overline{u}-z_{1}z_{2}\overline{u}^{2}$.$\ $ Then $f( u) =2i%
  \mathrm{Im}\left( \overline{z_{1}}\overline{z_{2}}u^{2}-\left( \overline{%
  z_{1}}+\overline{z_{2}}\right) u\right) $. Since $\overline{z_{1}}\overline{%
  z_{2}}h-\left( \overline{z_{1}}+\overline{z_{2}}\right) =\frac{2\overline{%
  z_{1}}\overline{z_{2}}\left( z_{2}-z_{1}\right) }{z_{1}\overline{z_{2}}-%
  \overline{z_{1}}z_{2}}$\,,
  it follows that
  \begin{equation*}
    \overline{z_{1}}\overline{z_{2}}h^{2}-\left( \overline{z_{1}}+\overline{z_{2}%
    }\right) h=\frac{-16\left\vert z_{2}-z_{1}\right\vert ^{2}}{\left\vert z_{1}%
    \overline{z_{2}}-\overline{z_{1}}z_{2}\right\vert ^{4}}\mathrm{Re}\left(
    z_{1}\overline{z_{2}}\right) \mathrm{Im}^{2}\left( z_{1}\overline{z_{2}}%
    \right)
  \end{equation*}
  is a real number, hence $f(h) =0\,.$
\end{proof}

Let $z_{1},z_{2}\in \mathbb{C}^{\ast }$ be such that
$\left\vert z_{1}\right\vert \neq \left\vert z_{2}\right\vert $ and
$\left\vert \arg z_{1}-\arg z_{2}\right\vert \notin \left\{ 0,\pi \right\} $\,.
Let $h$ be given by \eqref{eq:orthocenter}.
Note that $h-\left( \frac{1}{\overline{z_{1}}}+\frac{1}{\overline{z_{2}}}\right)
 =\frac{2\left(z_{2}-z_{1}\right) }{z_{1}\overline{z_{2}}-\overline{z_{1}}z_{2}}\neq 0$\,.
If $h\notin \left\{ 0,\frac{1}{\overline{z_{1}}},\frac{1}{\overline{z_{2}}}\right\} $
then the hyperbola $\Gamma $ passing through the five points
$0,\frac{1}{\overline{z_{1}}},\frac{1}{\overline{z_{2}}},
 \frac{1}{\overline{z_{1}}}+\frac{1}{\overline{z_{2}}}$\,,
$h$ can be constructed using a mathematical software.

In the cases where
$h\in \left\{ 0,\frac{1}{\overline{z_{1}}}, \frac{1}{\overline{z_{2}}}\right\} $\,,
we choose a vertex of the hyperbola $\Gamma $ as the fifth point needed to
construct $\Gamma $\,.
The vertices of the equilateral hyperbola $\Gamma $ are the intersections of
$\Gamma $ with the line passing through the center of the hyperbola, with the
slope $m=1$ if $\left\vert z_{1}\right\vert >\left\vert z_{2}\right\vert $\,,
respectively $m=-1$ if $\left\vert z_{1}\right\vert <\left\vert z_{2}\right\vert $\,.
Let $\alpha :=\frac{\arg z_{2}-\arg z_{1}}{2}$\,. Using \eqref{eq:conic3} it follows
that the distance $d$ between a vertex and the center of $\Gamma $ is
$d=\frac{\sqrt{\left\vert \left\vert z_{1}\right\vert ^{2}-\left\vert
z_{2}\right\vert ^{2}\right\vert }}{2\left\vert z_{1}z_{2}\right\vert }
\sqrt{\sin 2\alpha }$\,.

If $h=0$ we have $\alpha =\frac{\pi }{4}$ and
$d=\frac{1}{2}\sqrt{\Big| \big|
  \frac{1}{\overline{z_{2}}}\big|^{2}-\big| \frac{1}{\overline{z_{1}}}
  \big|^{2}\Big| }$\,.
Assume that $h=\frac{1}{\overline{z_{1}}}$,
the case $h=\frac{1}{\overline{z_{2}}}$ being similar.
Then $\left\vert z_{2}\right\vert =\left\vert z_{1}\right\vert
\cos 2\alpha <\left\vert z_{1}\right\vert $ and
$\left\vert\left\vert z_{1}\right\vert ^{2}-\left\vert z_{2}\right\vert ^{2}\right\vert
  =\left\vert z_{1}-z_{2}\right\vert ^{2}$, therefore $d=\frac{1}{2}\left\vert
   \frac{1}{\overline{z_{2}}}-\frac{1}{\overline{z_{1}}}\right\vert \sqrt{\sin
  2\alpha }$\,.
Let $z_{3}$ be the orthogonal projection of $\frac{1}{\overline{z_{1}}}$
on the line joining $\frac{1}{\overline{z_{2}}}$ to the origin.
Then
$d=\frac{1}{2}\sqrt{\left\vert \frac{1}{\overline{z_{2}}}
 -\frac{1}{\overline{z_{1}}}\right\vert \cdot \left\vert \frac{1}{\overline{z_{2}}}
 -z_{3}\right\vert }$\,.
We see that a vertex of $\Gamma $ can be constructed with ruler and compass
if $h\in \left\{ 0,\frac{1}{\overline{z_{1}}},\frac{1}{\overline{z_{2}}}\right\} $\,.

\begin{rem} \label{rem:myrem513}
  Being symmetric with respect to the center of $\Gamma$,
  $\frac{1}{\overline{z_{1}}}$ and $\frac{1}{\overline{z_{2}}}$ belong to
  distinct branches of $\Gamma $, each branch being divided by
  $\frac{1}{\overline{z_{1}}}$ or $\frac{1}{\overline{z_{2}}}$ into two arcs.
  If $z_{k}\in \mathbb{C\setminus }\overline{\mathbb{D}}$\,,
  $k\in \left\{1,2\right\} $\,, then each of these arcs joins
  $\frac{1}{\overline{z_{k}}}$\,,
  that is in the unit disk, with some point exterior to the unit disk,
  therefore it intersects the unit circle. It follows that, in the case of the
  exterior problem, $\Gamma $ intersects the unit circle at four distinct
  points.
\end{rem}

In the following we identify the points of intersection of the
conic section $\Gamma $ given by \eqref{eq:conic1} with the unit circle. After
finding the points $u\in \partial \mathbb{D}\cap \Gamma $ it is easy to
select among these the points $u$ for which \eqref{eq:theptu} holds,
respectively for which $\left\vert u-z_{1}\right\vert +\left\vert
u-z_{2}\right\vert $ attains its minimum or its maximum on
$\partial \mathbb{D}$\,.

First assume that $\Gamma $ is a pair of lines $d_{1},d_{2}$\,,
parallel to the interior bisector and to the exterior bisector of the
angle $\measuredangle (z_{1},0,z_{2})$, respectively.
Let $\alpha =\frac{1}{2}\left\vert \arg z_{2}-\arg z_{1}\right\vert $\,.
Then $\alpha \in \left\{ 0,\frac{\pi }{2}\right\} $ or
$\left\vert z_{1}\right\vert =\left\vert z_{2}\right\vert $\,.
The distances from the origin to $d_{1}$ and $d_{2}$ are
$\delta _{1}=\frac{\left\vert \left\vert z_{2}\right\vert -\left\vert
z_{1}\right\vert \right\vert }{2\left\vert z_{1}z_{2}\right\vert }\sin
\alpha $ and $\delta _{2}=\frac{\left\vert z_{1}\right\vert +\left\vert
z_{2}\right\vert }{2\left\vert z_{1}z_{2}\right\vert }\cos \alpha $\,.
Then $\Gamma $ intersects the unit circle at four distinct points in
the following cases:
(i) $z_{1},z_{2}\in \mathbb{C\setminus D}$;
(ii) $z_{1},z_{2}\in \mathbb{D}$ with
     $\frac{1}{2}\left\vert \frac{1}{\left\vert z_{1}\right\vert}
      -\frac{1}{\left\vert z_{2}\right\vert }\right\vert <1$
     or with $\left\vert z_{1}\right\vert
      =\left\vert z_{2}\right\vert >\cos \alpha $\,.
In the other cases for $z_{1},z_{2}\in \mathbb{D}$ the intersection of
$\Gamma $ with the unit circle consists of two distinct points.

\begin{prop} \label{prop:mmprop}
  If the conic section $\Gamma $ given by \eqref{eq:conic1} is
  a hyperbola, then the intersection of $\Gamma $ with the
  unit circle consists of
  \begin{enumerate}[{\rm (i)}]
    \item  four distinct points if
          $z_{1},z_{2}\in \mathbb{C\setminus D}$, one in the interior
          of each angle determined by the lines that pass
          through the origin and $z_{1}$, respectively $z_{2}$;
    \item at least two distinct points if $z_{1},z_{2}\in \mathbb{D}$,
          one in the interior of the angle determined by the rays passing
          starting at the origin and passing through $z_{1}$, respectively
          $z_{2}$ and the other in the interior of the opposite angle.
  \end{enumerate}
\end{prop}

\begin{proof}
  The intersection of $\Gamma $ with the unit circle consists of
  the points $u=e^{it}$, $t\in (-\pi ,\pi ]$ satisfying
  \begin{equation*}
    \mathrm{Im}\left( \overline{z_{1}z_{2}}e^{i2t}-\left( z_{1}+z_{2}\right)
    e^{-it}\right) =0\,.
  \end{equation*}%
  Let $z_{1},z_{2}\in \mathbb{C}^{\ast }$. There are at most four points of
  intersection of $\Gamma $ and the unit circle, since these are the
  roots of the quartic equation \eqref{eq:equation}.

  Using a rotation around the origin and a change of orientation
  we may assume that $\arg z_{2}=-\arg z_{1}=:\alpha $, where $0\leq \alpha
  \leq \frac{\pi }{2}\,.$ The above equation is equivalent to
  \begin{equation}\label{eq:teq}
    g\left( t\right) :=\left\vert z_{1}z_{2}\right\vert \sin 2t-\left\vert
    z_{1}\right\vert \sin \left( t+\alpha \right) -\left\vert z_{2}\right\vert
    \sin \left( t-\alpha \right) =0\,.
  \end{equation}

  We have
  \begin{equation*}
    g\left( -\pi \right) =g\left( \pi \right) =-g\left( 0\right) =\left(
    \left\vert z_{1}\right\vert -\left\vert z_{2}\right\vert \right) \sin
    \alpha\,,
  \end{equation*}
  \begin{equation*}
    g\left( \alpha -\pi \right) =\left\vert z_{1}\right\vert \left( \left\vert
    z_{2}\right\vert +1\right) \sin 2\alpha \,, \quad g\left( -\alpha \right)
    =\left\vert z_{2}\right\vert \left( 1-\left\vert z_{1}\right\vert \right)
    \sin 2\alpha \,,
  \end{equation*}
  \begin{equation*}
    g\left( \alpha \right) =\left\vert z_{1}\right\vert \left( \left\vert
    z_{2}\right\vert -1\right) \sin 2\alpha \,, \quad g\left( \pi -\alpha
    \right) =-\left\vert z_{2}\right\vert \left( \left\vert z_{1}\right\vert
    +1\right) \sin 2\alpha \,.
  \end{equation*}

  Consider the cases where $\Gamma $ is a hyperbola, i.e.
  $0<\alpha <\frac{\pi }{2}\,.$
  Clearly, $-\pi <\alpha -\pi <-\alpha <0<\alpha <\pi -\alpha <\pi $. We have
  $g\left( \pi -\alpha \right) <0<g\left( \alpha -\pi \right) $, while
  $g\left(-\pi \right) =g\left( \pi \right) =-g\left( 0\right) $ has the
  same sign as $\left\vert z_{1}\right\vert -\left\vert z_{2}\right\vert $\,.
  \begin{enumerate}[{\rm (i)}]
    \item Assume that $z_{1},z_{2}\in \mathbb{C\setminus D}$.
          Then $g\left( -\alpha \right) <0$ and $g\left( \alpha \right) >0$\,.

          If $\left\vert z_{1}\right\vert <\left\vert z_{2}\right\vert $\,,
          then $g\left( -\pi \right) <0$ $<g\left( \alpha -\pi \right) >0$\,,
          $g\left( -\alpha \right) <0<g\left( 0\right) $ and
          $g(\alpha )>0>g\left( \pi-\alpha \right) $.
          Since $g$ is continuous on $\mathbb{R}$, equation
          \eqref{eq:teq} has at least one root in each of the open intervals
          $\left( -\pi,\alpha -\pi \right) $, $\left( \alpha -\pi ,-\alpha \right) $,
          $\left(-\alpha ,0\right) $ and $\left( \alpha ,\pi -\alpha \right) $\,.

          If $\left\vert z_{2}\right\vert <\left\vert z_{1}\right\vert $\,,
          then $g\left( \alpha -\pi \right) >0>g\left( -\alpha \right) $\,,
          $g\left(0\right) <0$ $<g(\alpha )$ and
          $g\left( \pi -\alpha \right)<0<g\left( \pi \right) $\,.
          The equation \eqref{eq:teq} has at least one root in each of the open
          intervals $\left( \alpha -\pi ,-\alpha \right) $\,,
          $\left(0,\alpha \right) $\,,
          $\left( \alpha ,\pi -\alpha \right) $ and
          $\left( \pi-\alpha ,\pi \right) $.
    \item Now assume that $z_{1},z_{2}\in \mathbb{D}$.
          Then $g\left(-\alpha \right) >0$ and $g\left( \alpha \right) <0\,.$

          If $\left\vert z_{1}\right\vert <\left\vert z_{2}\right\vert $\,,
          then $g\left( -\pi \right) <0$  $<g\left( \alpha -\pi \right) $
          and $g\left( 0\right) >0$  $>g(\alpha )$\,.
          Since $g$ is continuous on $\mathbb{R}$\,, equation \eqref{eq:teq}
          has at least one root in each of the open intervals
          $\left( -\pi ,\alpha -\pi \right) $ and $\left( 0,\alpha \right) $\,.

          If $\left\vert z_{1}\right\vert >\left\vert z_{2}\right\vert $\,,
          then $g\left( 0\right) >0>g(\alpha )$ and
          $g\left( \pi -\alpha \right)<0<g\left( \pi \right) $.
          The equation \eqref{eq:teq} has at least one root in each of the
          open intervals $\left( 0,\alpha \right) $ and
          $\left( \pi -\alpha,\pi \right) $\,.
  \end{enumerate}
\end{proof}


\begin{figure}[t]
  \vspace{-0.1cm}
  \centerline{\includegraphics[width=0.5\linewidth]{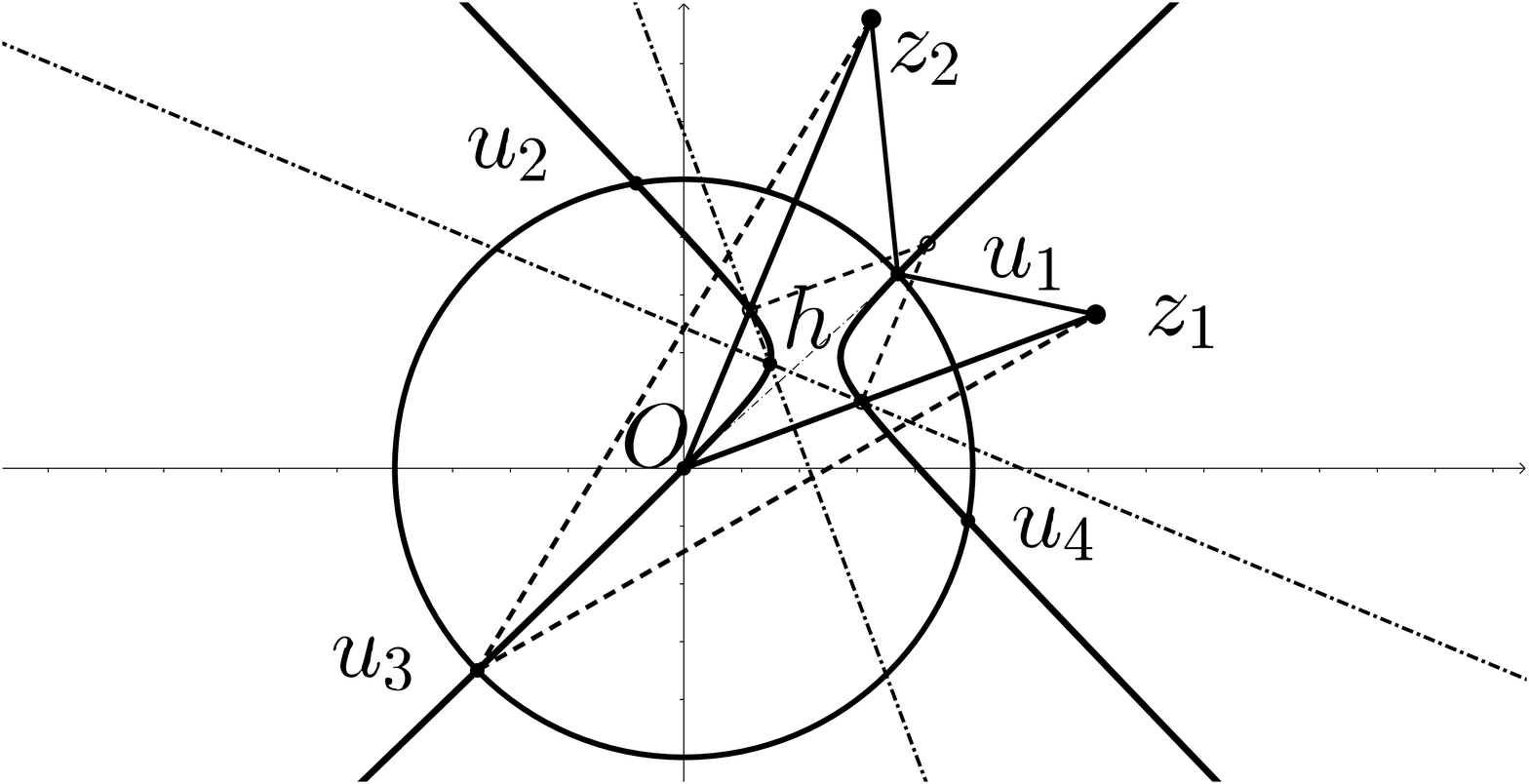}}
  \caption{The exterior problem. Intersection of the conic
           \eqref{eq:conic1} with the unit circle.}
\end{figure}

\begin{cor} \label{cor:mycor516}
  The equation \eqref{eq:equation} has four distinct unimodular
  roots in the case of the exterior problem and has at least two distinct
  unimodular roots in the case of the interior problem.
\end{cor}



\section{Remarks on the roots of the equation  \eqref{eq:equation}}
\label{section5}

In this section we study the number of the unimodular roots of the equation %
\eqref{eq:equation} (i.e., the roots lying on the unit circle) and their
multiplicities.
Denote $P(u)=\overline{z_{1}}\overline{z_{2}}u^{4}
 -(\overline{z_{1}}+\overline{z_{2}})u^{3}+(z_{1}+z_{2})u-z_{1}z_{2}$\,.
If either $z_{1}=0$ or $z_{2}=0$ then the cubic equation
\eqref{eq:equation} $P\left( u\right) =0$ has a root $u=0$ and two simple
roots on the unit circle.

We will assume in the following that $z_{1}\neq 0$ and $z_{2}\neq 0$\,.
As we observed in Section 2,
the quartic polynomial $P$ is self-inversive. Then $P$ has
an even number of zeros on the unit circle, each zero being counted as many
times as its multiplicity. According to Lemma \ref{lem:mylem28} $P$ has at
least two unimodular zeros, distinct or not, that is $P$ has four or two
unimodular zeros. There is a rich literature dealing with the location
of zeros of a complex self-inversive polynomial with respect to the unit
circle. After the publication of \cite{cohn}, also many other papers on this topic were published, see \cite{bm}, \cite{ck}, \cite{ls1}, \cite{ls2}, \cite{marden}.

Recall that the celebrated Gauss-Lucas theorem shows that the zeros of the
derivative $P^{\prime }$ of a complex polynomial $P$ lie within the convex hull
of the set of zeros of $P$. If a complex polynomial $P$ has all its zeros on
the unit circle, then the polynomial is self-inversive and, according to
Gauss-Lucas theorem \cite[Thm 6.1]{marden} all the zeros of $P^{\prime }$ are in the closed unit
disk. Moreover, the converse holds. A theorem of Cohn \cite{cohn} states
that a complex polynomial has all its zeros on the unit circle if and only
if the polynomial is self-inversive and its derivative has all its zeros in
the closed unit disk. A refinement of Cohn's theorem \cite[Theorem 1]{ck}
proves that all the zeros of a self-inversive polynomial $P(z)$\ lie on the
unit circle and are simple if and only if there exists a polynomial $Q(z)$\
with all its zeros in the unit disk $\left\vert z\right\vert <1$\
such that
$ P\left( z\right) =z^{m}Q\left( z\right) +e^{i\theta }Q^{\ast }\left(z\right) $
for some nonnegative integer $m$ and real $\theta $\,, where
$Q^{\ast }\left( z\right) =z^{n}\overline{Q}\left( \frac{1}{z}\right) $\,,
where $n=\deg Q$. A lemma in \cite{bm} shows that each unimodular zero of
the derivative of a self-inversive polynomial $P$ is also a zero of $P$\,.

\begin{lem}
  $P(u)=\overline{z_{1}}\overline{z_{2}}u^{4}
    -(\overline{z_{1}}+\overline{z_{2}})u^{3}+(z_{1}+z_{2})u-z_{1}z_{2}$
  cannot have two double zeros on the unit circle.
\end{lem}

\begin{proof}
  Assume that $P$ has two double zeros $a$ and $b$ on the unit circle,
  $P(u)= \overline{z_{1}z_{2}}(z-a)^{2}(z-b)^{2}\
     (a,b\in \partial \mathbb{D},\ a\neq b)$\,.
  Since the coefficient of $u^{2}$ in $P\left( u\right) $ vanishes,
  \begin{equation*}
    a^{2}+4ab+b^{2}=\big(a+(2-\sqrt{3})b\big)\big(a+(2+\sqrt{3})b\big)=0\,.
  \end{equation*}%
  This contradicts the assumption $|a|=|b|=1$\,.
\end{proof}

Similarly, we rule out another case.

\begin{lem} \label{lem:doubleuni}
  For $P(u)=\overline{z_{1}}\overline{z_{2}}u^{4}
       -(\overline{z_{1}}+\overline{z_{2}})u^{3}+(z_{1}+z_{2})u-z_{1}z_{2}$
  it is not possible to have a double zero on the unit circle and
  two zeros not on the unit circle.
\end{lem}

\begin{proof}
  Assume that $P$ has a double zero $a$ with $|a|=1$ and the zeros
  $b\neq\frac{1}{\overline{b}}$\,.
  Then $P(u)=\overline{z_{1}z_{2}}(z-a)^{2}(z-b)
  \left( z-\frac{1}{\overline{b}}\right) $\,.
   The coefficient of $u^{2}$ in $P\left( u\right) $ vanishes,
  \begin{equation*}
    a^{2}+\frac{b}{\overline{b}}+2a\left( b+\frac{1}{\overline{b}}\right)
      =0\,.
  \end{equation*}

  We have
  \begin{equation*}
    \left\vert b+\frac{1}{\overline{b}}\right\vert ^{2}
       =\left( b+\frac{1}{\overline{b}}\right) \left( \overline{b}
          +\frac{1}{b}\right) =2+\left\vert b\right\vert ^{2}
          +\frac{1}{\left\vert b\right\vert ^{2}}
       >4\,.
  \end{equation*}%
  Then $2\geq \left\vert a^{2}+\frac{b}{\overline{b}}\right\vert
         =\left\vert 2a\left( b+\frac{1}{\overline{b}}\right) \right\vert
         >4$\,, a contradiction.
\end{proof}

\begin{lem}
  If $P(u)=\overline{z_{1}}\overline{z_{2}}u^{4}
     -(\overline{z_{1}}+\overline{z_{2}})u^{3}+(z_{1}+z_{2})u-z_{1}z_{2}$
  has a triple zero $a$ and a simple
  zero $b$, then $b=-a$, with $a$ and $b$ lying on the unit circle and
  $\left\vert z_{1}+z_{2}\right\vert =2\left\vert z_{1}z_{2}\right\vert $\,.
\end{lem}

\begin{proof}
  Assume that $P$ has a triple zero $a$ and a simple zero $b$,
  $P(u)=\overline{z_{1}z_{2}}(z-a)^{3}(z-b)$,
  where $a,b\in \mathbb{C},\ a\neq b$. Since $P$
  is self-inversive, $|a|=|b|=1$ and $b=\frac{1}{\overline{a}}=-a$\,.
  Also, the fact that the coefficient of $u^{2}$ in $P(u)$ vanishes
  already implies $a(a+b)=0$.
  But $\overline{z_{1}z_{2}}a^{2}b=-z_{1}z_{2}\neq 0$,
  therefore $b=-a$.
  Considering the coefficient of $u^{3}$ in
  $P(u)=\overline{z_{1}z_{2}}(u-a)^{3}(u+a)$,
  it follows that
  $2a\overline{z_{1}}\overline{z_{2}}=\overline{z_{1}}+\overline{z_{2}}$,
  hence
  $\left\vert z_{1}+z_{2}\right\vert=2\left\vert z_{1}z_{2}\right\vert $\,.
\end{proof}

\begin{example}
  Find the relation between $z_{1},z_{2}$ such that
  $P(u)=\overline{z_{1}}\overline{z_{2}}u^{4}-(\overline{z_{1}}
    +\overline{z_{2}})u^{3}+(z_{1}+z_{2})u-z_{1}z_{2}$
  has the triple zero $1$ and the simple zero $(-1)$.

  Suppose
  \begin{equation}\label{eq:UU}
    P(u)=\overline{z_{1}z_{2}}(u-1)^{3}(u+1)=0\,.
  \end{equation}%
  From the constant term of \eqref{eq:equation} and \eqref{eq:UU},
  we have $z_{1}z_{2}\in \mathbb{R}$.
  Similarly, from the coefficient of $u$ in \eqref{eq:equation} and
  \eqref{eq:UU}, we have
  \begin{equation*}
    z_{1}+z_{2}-2z_{1}z_{2}=0\,.
  \end{equation*}%
  Therefore $z_{1}$ and $z_{2}$ coincide with the two solutions of
  $w^{2}-2pw+p=0$, where $p=z_{1}z_{2}\in \mathbb{R}$
  $($in particular $-1<p<1$ for the interior problem$)$.

  In the case where $0<p<1$, $z_{1}$ and $z_{2}$ are complex conjugates to
  each other since $\mbox{\rm discriminant}(w^{2}-2pw+p,w)=4(p^{2}-p)<0$.
  Hence, $P(u)=z_{1}\overline{z_{1}}(u-1)^{3}(u+1)=0$\,,
  and we have
  \begin{equation*}
    (2\overline{z_{1}}-1)z_{1}-\overline{z_{1}}
     =2\Big|z_{1}-\frac{1}{2}\Big|-\frac{1}{2}=0\,.
  \end{equation*}%
  Therefore, for $z_{1}$ on the circle $|z-\frac{1}{2}|=\frac{1}{2}$
  and $ z_{2}=\overline{z_{1}}$, $P(u)=0$ has exactly two roots $1$ and $-1$. This case was studied in \cite[Thm 3.1]{hkvz}.
  In fact, for $z_{1}=a+bi$ with $a^{2}-a+b^{2}=0$, $P(u)=a(u-1)^{3}(u+1)=0$\,.

  In the case where $-1<p<0$, the quadratic equation $w^{2}-2pw+p=0$ has two
  real roots and we have
  \begin{equation*}
    P(u)=z_{1}z_{2}(u-1)^{3}(u+1)\,.
  \end{equation*}%
  Moreover, we can parametrize two foci as follows,
  $z_{1}=t,\ z_{2}=\frac{t}{2t-1}$ $(-1<t<\sqrt{2}-1)$\,.
\end{example}

It remains to study the following cases:

\begin{enumerate}[{\bf {Case} 1.}]
   \item $P$ has four simple unimodular zeros.
   \item  $P$ has two simple unimodular zeros and two zeros that are
          not unimodular.
   \item $P$ has a double unimodular zero and two simple unimodular
         zeros.
\end{enumerate}

\begin{prop}\label{prop:foursimple}
  Assume that $z_{1},z_{2}\in \mathbb{C}^{\ast }$.
   Let $P(u)=\overline{z_{1}}\overline{z_{2}}u^{4}
    -(\overline{z_{1}}+\overline{z_{2}})u^{3}+(z_{1}+z_{2})u-z_{1}z_{2}$\,.
   Then
   \begin{enumerate}[{\rm a)}]
     \item  \label{item:a}
            $P$ has four simple unimodular zeros if $\left\vert
            z_{1}+z_{2}\right\vert <\left\vert z_{1}z_{2}\right\vert $ and
     \item \label{item:b}
           $P$ has exactly two unimodular zeros, that are simple,
           if $\left\vert z_{1}+z_{2}\right\vert
                 >2\left\vert z_{1}z_{2}\right\vert $.
    \item  \label{item:c}
           If $P$ has four simple unimodular zeros, then $\left\vert
            z_{1}+z_{2}\right\vert <2\left\vert z_{1}z_{2}\right\vert $.
    \item \label{item:d}
          If $P$ has exactly two unimodular zeros, that are simple,
          then $\left\vert z_{1}+z_{2}\right\vert
           >\left\vert z_{1}z_{2}\right\vert $.
  \end{enumerate}
\end{prop}

\begin{proof}
  We have $P^{\prime }\left( u\right) =4\overline{z_{1}z_{2}}u^{3}-3\left(
  \overline{z_{1}}+\overline{z_{2}}\right) u^{2}+\left( z_{1}+z_{2}\right) $
  and $P^{\prime \prime }\left( u\right) =12\overline{z_{1}z_{2}}u^{2}-6\left(
  \overline{z_{1}}+\overline{z_{2}}\right) u$.
  \begin{enumerate}[{\rm a)}]
    \item  Assume that $\left\vert z_{1}+z_{2}\right\vert <\left\vert
           z_{1}z_{2}\right\vert $. Then for $u\in \partial \mathbb{D}$
           we have
          \begin{equation*}
             \left\vert 4\overline{z_{1}z_{2}}u^{3}\right\vert
               =4\left\vert z_{1}z_{2}\right\vert
               >4\left\vert z_{1}+z_{2}\right\vert
               \geq \left\vert-3\left( \overline{z_{1}}
                     +\overline{z_{2}}\right) u^{2}+\left(z_{1}+z_{2}\right)
                     \right\vert \,
          \end{equation*}%
          It follows by Rouch\'{e}'s theorem \cite[3.10]{sg} that the derivative
          $P^{\prime }$ has all its zeros in the unit disk.
          By Cohn's theorem \cite{cohn}, $P$ has all its
          four zeros on the unit circle $\partial \mathbb{D}$.

          Moreover, $P\left( u\right) =u^{m}Q\left( z\right)
         +e^{i\theta }Q^{\ast}\left( u\right) $ for $m=3$,
         $\theta =\pi $ and
         $Q\left( u\right) =\overline{z_{1}z_{2}}u^{3}
            +\left( z_{1}+z_{2}\right) $.
         The roots of $Q$ have modulus
         $\sqrt[3]{\frac{\left\vert z_{1}+z_{2}\right\vert }{\left\vert
         z_{1}z_{2}\right\vert }}$.
          If $\left\vert z_{1}+z_{2}\right\vert
              <\left\vert z_{1}z_{2}\right\vert $,
         Theorem 1 from \cite{ck} shows that $P$ has four
         simple zeros on the unit circle.
    \item Now assume that
          $\left\vert z_{1}+z_{2}\right\vert >2\left\vert
          z_{1}z_{2}\right\vert $.
          For $u\in \partial \mathbb{D}$ we have
          \begin{equation*}
            \left\vert -3\left( \overline{z_{1}}+\overline{z_{2}}\right)
            u^{2}\right\vert =3\left\vert z_{1}+z_{2}\right\vert >4\left\vert
            z_{1}z_{2}\right\vert +\left\vert z_{1}+z_{2}\right\vert
            \geq \left\vert 4%
            \overline{z_{1}z_{2}}u^{3}+\left( z_{1}+z_{2}\right) \right\vert
          \end{equation*}%
          and it follows using Rouch\'{e}'s theorem that $P^{\prime }$
          has exactly two zeros in the closed unit disk.
          Cohn's theorem shows that $P$ cannot have all
          its zeros on $\partial \mathbb{D}$.
          By Lemma \ref{lem:mylem28}, $P$ has at
          least two unimodular zeros, therefore $P$ has exactly
          two unimodular zeros.
          By Lemma \ref{lem:doubleuni}, these unimodular zeros are simple.

          An alternative way to prove that $P$ has exactly
          two unimodular zeros is indicated below.
          Assume by contrary that $P$ has four unimodular zeros.
          Using the Gauss-Lucas theorem two times, it follows that
          each of the derivatives $P^{\prime }$and
          $P^{\prime \prime }$ has all its zeros in the closed unit disc
          $|z|\leq 1$.
          The zeros of $P^{\prime \prime }$ are $0$ and
          $\frac{\overline{z_{1}}+\overline{z_{2}}}{2\overline{z_{1}z_{2}}}$.
          Then, under the assumption
          $\left\vert z_{1}+z_{2}\right\vert
            >2\left\vert z_{1}z_{2}\right\vert $, the
          second derivative $P^{\prime \prime }$ has a zero in $|z|>1$, which is a
          contradiction.
    \item Assume that $P$ has four simple unimodular zeros.
          Then $P^{\prime }$ has all its zeros in the closed unit disk.
          If $P^{\prime }$ has a unimodular zero $a$,
          then $P\left( a\right) =0$ according to \cite{bm}, therefore $a$
          is a zero of $P$ of multiplicity at least $2$, a contradiction.
          It follows that $P^{\prime }$ has all its zeros in the unit disk.
          By Gauss-Lucas theorem, $P^{\prime \prime }$ also has all its zeros
          in the unit disk,
          therefore $\left\vert z_{1}+z_{2}\right\vert <2\left\vert
          z_{1}z_{2}\right\vert $.
     \item Now suppose that $P$ has exactly two simple unimodular zeros,
           $a$ and $b$.
           Let $c$ and $\frac{1}{\overline{c}}$ the other zeros of $P$,
           with $\left\vert c\right\vert <1$.
           Then $P\left( u\right) =\overline{z_{1}}\overline{z_{2}}
           \left( u-a\right) \left( u-b\right) \left( u-c\right) \left(
           u-\frac{1}{\overline{c}}\right) $.
           The coefficient of $u^{2}$ in $P\left(u\right) $ vanishes,
           therefore
           \begin{equation*}
              ab+\frac{c}{\overline{c}}+\left( a+b\right)
                 \left( c+\frac{1}{\overline{c}}\right) =0,
           \end{equation*}%
           and $a+b=-\frac{ab}{c+\frac{1}{\overline{c}}}-\frac{c}{\left\vert
           c\right\vert ^{2}+1}$.
           Because $\left\vert \frac{ab}{c+\frac{1}{\overline{c}}}\right\vert
             =\frac{1}{\left\vert c+\frac{1}{\overline{c}}\right\vert }
             <\frac{1}{2}$ and
           $\frac{\left\vert c\right\vert }{\left\vert c\right\vert ^{2}+1}
             <\frac{1}{2}$, we get
           $\left\vert a+b\right\vert <1$. Considering the
           coefficient of $u^{3}$ in $P\left( u\right) $ we obtain
           $\frac{\overline{z_{1}}+\overline{z_{2}}}{\overline{z_{1}z_{2}}}
            =a+b+c+\frac{1}{\overline{c}}$.
           Then $\frac{\left\vert z_{1}+z_{2}\right\vert }
                      {\left\vert z_{1}z_{2}\right\vert }
                \geq \left\vert \left\vert c+\frac{1}{\overline{c}}
                     \right\vert -\left\vert a+b\right\vert \right\vert
                >1$.
  \end{enumerate}
\end{proof}

\begin{example}
  Let $z_{1}=\left( 1+t\right) e^{i\alpha }$ and
  $z_{2}=(1+t)e^{i\left( \alpha+t\right) }$, where $t>0$ and
  $\alpha \in (-\pi ,\pi ]$. By Corollary {\rm \ref{cor:mycor516}},
  the equation \eqref{eq:equation} has four simple unimodular roots
  in this case. On the other hand,
  $\frac{\left\vert z_{1}+z_{2}\right\vert }
        {\left\vert z_{1}z_{2}\right\vert }
    =\left( 1+t\right) \left( 1+e^{-it}\right)\rightarrow 2$
  as $t\rightarrow 0$, therefore the constant $2$ in
  Proposition {\rm \ref{prop:foursimple} \ref{item:c})}
  cannot be replaced by a  smaller constant.
\end{example}

We give a direct proof for the following consequence of Proposition
\ref{prop:foursimple}.

\begin{cor}
  If $P(u)=\overline{z_{1}}\overline{z_{2}}u^{4}-(\overline{z_{1}}
       +\overline{z_{2}})u^{3}+(z_{1}+z_{2})u-z_{1}z_{2}$
  has one double zero and two simple zeros on the unit circle,
  then $\left\vert z_{1}z_{2}\right\vert \leq
  \left\vert z_{1}+z_{2}\right\vert \leq 2\left\vert z_{1}z_{2}\right\vert $.
\end{cor}

\begin{proof}
  Assume that $P$ has one double unimodular zero $a$ and two simple
  unimodular zeros $b,c$.
  Then $P\left( u\right) =\overline{z_{1}}\overline{z_{2}}\left(
   z-a\right) ^{2}\left( z-b\right) \left( z-c\right) $.

  The coefficient of $u^{2}$ in $P\left( u\right) $ vanishes,
  \begin{equation*}
    a^{2}+bc+2a\left( b+c\right) =0.
  \end{equation*}%
  Considering the coefficient of $u^{3}$ in $P\left( u\right) $ we obtain
  $\frac{\overline{z_{1}}+\overline{z_{2}}}{\overline{z_{1}z_{2}}}=2a+b+c=2a
   -\frac{a^{2}+bc}{2a}=\frac{3a^{2}-bc}{2a}=\frac{3}{2}a-\frac{bc}{2a}$.
  Then $\frac{\left\vert z_{1}+z_{2}\right\vert }
             {\left\vert z_{1}z_{2}\right\vert }
         \leq \left\vert \frac{3}{2}a\right\vert +\left\vert
               -\frac{bc}{2a}\right\vert =2$
  and $\frac{\left\vert z_{1}+z_{2}\right\vert }
            {\left\vert z_{1}z_{2}\right\vert }
         \geq \left\vert \left\vert \frac{3}{2}a\right\vert
            -\left\vert -\frac{bc}{2a}\right\vert \right\vert =1$.
\end{proof}


\bigskip

\textbf{Acknowledgments.} This research was begun
during the Roma\-nian-Finnish Seminar in Bucharest, Romania, June 20-24,
2016, where the authors P.H., M.M., and M.V. met. During a workshop at the
Tohoku University, Sendai, Japan, in August 2016 organized by Prof. T.
Sugawa,  M.F., P.H., and M.V. met and had several discussions about the
topic of this paper. P. H. and M. V. are indebted to Prof. Sugawa for kind
and hospitable arrangements during our visit. This work was partially supported by JSPS KAKENHI
    Grant Number 15K04943. The second author was supported by University of Turku Foundation and CIMO. The authors are indebted to
Prof. G.D. Anderson for a number of remarks on this paper.

\nocite{aa}
\nocite{pac}
\nocite{hkvz}

\bibliographystyle{siamplain}
\bibliography{PAref6}


\end{document}